\documentclass[11 pt, psamsfonts,  leqno]{amsart}
\usepackage{amssymb}
\usepackage{amsmath}	

\def\ignore#1{\relax}

\usepackage[utopia]{mathdesign}

\usepackage{bm}

 \calclayout

\usepackage[plainpages=false, pdfpagelabels,  colorlinks=true, pdfstartview=FitV, linkcolor=blue, citecolor=blue, urlcolor=blue]{hyperref}

\usepackage{mdwlist}

\numberwithin{equation}{section}
\numberwithin{figure}{section}

\def\Z{{\mathbb Z}}

\def\Q{{\mathbb Q}}

\def\eps{\varepsilon}
 
\def\u #1 #2{\mathcal U(#1, #2)}  
\def\uhat #1 #2{\widehat{\mathcal U}(#1, #2)}  

\def\bmw #1{W_{#1}}  

\def\w #1 #2{\bmw {#1}^{(#2)}}  
\def\V #1 #2{V_{#1}^{(#2)}}  
\def\k #1 #2{ KT_{#1}^{(#2)}}
 
\def\inv{^{-1}}

\def\rhobold{{\bm \rho}}

\def\qbold{{\bm q}}
\def\ubold{{\bm u}}

\def\powerpm{^{\pm 1}}

\def\hods{\unskip\kern.55em\ignorespaces}

\theoremstyle{plain}
\newtheorem{theorem}{Theorem}[section]

\theoremstyle{plain}

\theoremstyle{plain}
\newtheorem{corollary}[theorem]{Corollary}

\theoremstyle{plain}
\newtheorem{lemma}[theorem]{Lemma}

\theoremstyle{definition}
\newtheorem{definition}[theorem]{Definition}
\theoremstyle{definition}

\theoremstyle{definition}
\newtheorem{remark}[theorem]{Remark}

\theoremstyle{remark}

\title[Cyclotomic BMW algebras]{Comparison of admissibility conditions for Cyclotomic Birman--Wenzl--Murakami algebras}

\author{Frederick M. Goodman}
\address{ Department of Mathematics\\ University of Iowa\\ Iowa
City, Iowa}
\email{ goodman@math.uiowa.edu}


\begin{document}
 To appear in {\em Journal of Pure and Applied Algebra}
 \bigskip
 
 \begin{abstract}

 We show the equivalence of admissibility conditions proposed by Wilcox and Yu ~\cite{Wilcox-Yu} and by Rui and Xu ~\cite{rui-2008}  for the 
 parameters of cyclotomic BMW algebras.
 \end{abstract}

 \maketitle

 \section{Introduction}
Cyclotomic Birman--Wenzl--Murakami (BMW)  algebras are BMW analogues of cyclotomic Hecke algebras  ~\cite{ariki-koike, ariki-book}.  They were defined by 
H\"aring--Oldenburg in ~\cite{H-O2}  and have recently been studied by three groups of mathematicians:
Goodman and  \break Hauschild--Mosley ~\cite{GH1, GH2, GH3,  goodman-2008},  Rui, Xu, and Si  ~\cite{rui-2008, rui-2008b},   and Wilcox and Yu  ~\cite{Wilcox-Yu, Wilcox-Yu2, Wilcox-Yu3, Yu-thesis}.

A peculiar feature of these algebras is that it is necessary to impose ``admissibility" conditions on the parameters entering into the definition of the algebras in order to obtain a satisfactory theory.   There is no one obvious best set of conditions,  and the different groups studying these algebras have proposed different admissibility conditions and have chosen slightly different settings for their work. 

Under their various admissibility hypotheses on the ground ring, the  several groups of mathematicians  mentioned above have obtained similar results for the cyclotomic BMW algebras, namely  freeness and cellularity.  In addition,  Goodman \&  Hauschild--Mosley and Wilcox  \& Yu  have shown that the algebras can be realized as algebras of tangles, while
Rui et. al.  have obtained additional representation theoretic results,  for example, classification of simple modules and semisimplicity criteria.  However,  it has been difficult to compare the results of the different investigations because of the different settings.

The purpose of this note is to show that the admissibility condition proposed by Rui and Xu  ~\cite{rui-2008} is equivalent to the condition proposed by Wilcox and Yu  ~\cite{Wilcox-Yu}.
As a result, one has a consensus setting for the study of cyclotomic BMW algebras.

Further background on cyclotomic BMW algebras, motivation for the study of these algebras,  relations to other mathematical topics (quantum groups, knot theory), and further literature citations  can be found in ~\cite{GH2} and in the other papers cited above.

\subsection*{Acknowledgements}  I would like to thank Hebing Rui and Shona Yu for stimulating email correspondence about the topic of this paper.  I thank the referee for helpful comments which led to an improvement of the exposition.  

\newpage
\section{Definitions}

In general we use the definitions and notation from  ~\cite{GH3}.

\begin{definition} Fix an integer $r \ge 1$.   A {\em ground ring} 
$S$ is a commutative unital ring with parameters $\rho$, $q$, $\delta_j$   ($j \ge 0$),  and
$u_1, \dots, u_r$, with    $\rho$, $q$,   and $u_1, \dots, u_r$  invertible,  and with 
\begin{equation} \label{equation:  basic relation in ground ring}
\rho\inv - \rho=   (q\inv -q) (\delta_0 - 1).
\end{equation}
\end{definition}

\begin{definition} \label{definition:  cyclotomic BMW}
Let $S$ be a ground ring with
parameters $\rho$, $q$, $\delta_j$   ($j \ge 0$),  and
$u_1, \dots, u_r$.
The {\em cyclotomic BMW algebra}  $\bmw{n, S, r}(u_1, \dots, u_r)$  is the unital $S$--algebra
with generators $y_1^{\pm 1}$, $g_i^{\pm 1}$  and
$e_i$ ($1 \le i \le n-1$) and relations:
\begin{enumerate}
\item (Inverses)\hods $g_i g_i\inv = g_i\inv g_i = 1$ and 
$y_1 y_1\inv = y_1\inv y_1= 1$.
\item (Idempotent relation)\hods $e_i^2 = \delta_0 e_i$.
\item (Affine braid relations) 
\begin{enumerate}
\item[\rm(a)] $g_i g_{i+1} g_i = g_{i+1} g_ig_{i+1}$ and 
$g_i g_j = g_j g_i$ if $|i-j|  \ge 2$.
\item[\rm(b)] $y_1 g_1 y_1 g_1 = g_1 y_1 g_1 y_1$ and $y_1 g_j =
g_j y_1 $ if $j \ge 2$.
\end{enumerate}
\item[\rm(4)] (Commutation relations) 
\begin{enumerate}
\item[\rm(a)] $g_i e_j = e_j g_i$  and
$e_i e_j = e_j e_i$  if $|i-
j|
\ge 2$. 
\item[\rm(b)] $y_1 e_j = e_j y_1$ if $j \ge 2$.
\end{enumerate}
\item[\rm(5)] (Affine tangle relations)\vadjust{\vskip-2pt\vskip0pt}
\begin{enumerate}
\item[\rm(a)] $e_i e_{i\pm 1} e_i = e_i$,
\item[\rm(b)] $g_i g_{i\pm 1} e_i = e_{i\pm 1} e_i$ and
$ e_i  g_{i\pm 1} g_i=   e_ie_{i\pm 1}$.
\item[\rm(c)\hskip1.2pt] For $j \ge 1$, $e_1 y_1^{ j} e_1 = \delta_j e_1$. 
\vadjust{\vskip-
2pt\vskip0pt}
\end{enumerate}
\item[\rm(6)] (Kauffman skein relation)\hods  $g_i - g_i\inv = (q - q\inv)(1- e_i)$.
\item[\rm(7)] (Untwisting relations)\hods $g_i e_i = e_i g_i = \rho \inv e_i$
 and $e_i g_{i \pm 1} e_i = \rho  e_i$.
\item[\rm(8)] (Unwrapping relation)\hods $e_1 y_1 g_1 y_1 = \rho e_1 = y_1 
g_1 y_1 e_1$.
\item[\rm(9)](Cyclotomic relation) \hods $(y_1 - u_1)(y_1 - u_2) \cdots (y_1 - u_r) = 0$.
\end{enumerate}
\end{definition}

Thus,  a cyclotomic BMW algebra is the quotient of the affine BMW algebra  ~\cite{H-O2, GH1}, by the cyclotomic relation $(y_1 - u_1)(y_1 - u_2) \cdots (y_1 - u_r) = 0$.  We recall from ~\cite{GH1} that the affine BMW algebra is isomorphic to an algebra of framed affine tangles, modulo Kauffman skein relations.  Assuming admissible parameters, it has been shown that the cyclotomic BMW algebras are also isomorphic to tangle algebras ~\cite{GH3, Wilcox-Yu3, Yu-thesis}.

\begin{lemma}  \label{lemma - recursion for f_r}
 For $j \ge 1$,  there exist elements $\delta_{-j} \in \Z[\rho^{\pm 1}, q^{\pm 1}, \delta_0, \dots, \delta_j]$ such that \break $e_1 y_1^{-j} e_1 = \delta_{-j} e_1$.  Moreover,  the elements $\delta_{-j}$  are determined by the recursion relation:
\begin{equation} \label{equation:  delta(-j)  recursive relations}
\begin{aligned}
\delta_{-1} &= \rho^{-2} \delta_1 \\
\delta_{-j}  & =  \rho^{-2}  \delta_j +  (q\inv - q) \rho\inv  \sum_{k = 1}^{j-1} (\delta_k \delta_{k-j} - \delta_{2k -j}   ) \quad  (j \ge 2).
\end{aligned}
\end{equation}
\end{lemma}

\begin{proof}  Follows from ~\cite{GH1},   Corollary 3.13, and
~\cite{GH2}, Lemma 2.6;  or  ~\cite{rui-2008}, Lemma 2.17.
\end{proof}

 We consider what are the appropriate morphisms between ground rings for cyclotomic BMW algebras.   The obvious notion would be that of a ring homomorphism taking parameters to parameters;  that is,  if $S$ is a ground ring with parameters $\rho$, $q$, 
etc.,  and $S'$ another ground  ring with parameters $\rho'$, $q'$, etc.,  then a morphism $\varphi : S \to S'$  would be required to map $\rho \mapsto \rho'$,  
$q \mapsto q'$,  etc.  

However,  it is better to require less, for the following reason:  The parameter $q$ enters into the cyclotomic BMW
 relations only in the expression $q\inv -q$, and the transformation $q \mapsto -q\inv$ leaves this expression invariant.  Moreover,  the transformation $g_i  \mapsto -g_i$,  $\rho \mapsto -\rho$,  $q \mapsto -q$ (with all other generators and parameters unchanged)   leaves the cyclotomic BMW relations unchanged.

Taking this into account, we arrive at the following notion:

\begin{definition}  \label{definition: parameter preserving}
Let $S$ be a ground ring with
parameters $\rho$, $q$, $\delta_j$   ($j \ge 0$),  and
$u_1, \dots, u_r$.
Let  $S'$ be another ground   ring with parameters $\rho'$, $q'$, etc.  

A unital ring homomorphism $\varphi : S \rightarrow S'$ is a {\em morphism of ground rings}  if it maps
$$
\begin{cases}
&\rho \mapsto \rho',  \text{ and}\\
& q \mapsto q' \text{ or } q \mapsto    -{q'}\inv,
\end{cases}
$$
or
$$
\begin{cases}
&\rho \mapsto  -\rho',  \text{ and}\\
& q \mapsto -q' \text{ or } q \mapsto    {q'}\inv,
\end{cases}
$$
and strictly preserves all other parameters.
\end{definition}

 Suppose there is a morphism of ground rings $\psi : S \rightarrow S'$.
  Then $\psi$ extends to a 
homomorphism  from 
$\bmw{n, S, r}$  to $\bmw{n, S', r}$.     Moreover,  $\bmw{n, S, r} \otimes_S S' \cong \bmw{n,S', r}$  as $S'$--algebras.   These statements are discussed in ~\cite{GH3},  Section 2.4.

\section{Admissibility conditions}

The following {\em weak admissibility} condition is a minimal condition on the parameters to obtain a non--trivial algebra;  in the absence of weak admissibility, the generator $e_1$ is a torsion element over the ground ring; if $S$ is a field, then $e_1 = 0$,  and the cyclotomic BMW algebra reduces to a specialization of the cyclotomic Hecke algebra.   See the remarks preceding Definition 2.14 in ~\cite{GH3}.   

In the following definition,
$a_j$ denotes the signed elementary symmetric function in $u_1, \dots, u_r$,  namely,
$
a_j = (-1)^{r-j} \varepsilon_{r-j}(u_1, \dots, u_r).  
$

\begin{definition} \label{definition: weak admissibility}
Let $S$ be a ground ring with parameters
$\rho$, $q$,  $\delta_j$, $j \ge 0$, and $u_1, \dots, u_r$.
  We say that the parameters are {\em weakly admissible}  (or that the ring $S$ is weakly
admissible)  if the  following relation holds:
$$
\sum_{k = 0}^r a_k  \delta_{k + a}  = 0,
$$
for $a \in \Z$,  where for $j \ge 1$,   $\delta_{-j}$ is defined by  the recursive relations of Lemma \ref{lemma - recursion for f_r}.
\end{definition}

In order to obtain substantial results on the cyclotomic BMW algebras, it appears necessary to impose a condition on the ground ring that is stronger than weak admissibility.   
Two conditions have been proposed, one by Wicox and Yu,  and another by Rui and Xu.

First we consider the admissibility condition of Wilcox and Yu.  
Consider a ground  ring $S$ with parameters $\rho$, $q$, $\delta_j$  ($j \ge 0$)  and $u_1, \dots, u_r$. 
Let
 $W_2$  denote the cyclotomic BMW algebra
$W_2 = \bmw {2, S, r}(u_1, \dots, u_r)$.

\begin{lemma} \label{lemma: powers of y generate W2e}   The left ideal $W_2\, e_1$ in $W_2$  is equal to the $S$--span of  $\{e_1, y_1 e_1, \dots, y_1^{r-1} e_1\}$.
\end{lemma}

\begin{theorem}[Wilcox \& Yu, \cite{Wilcox-Yu}] \label{theorem: equivalent conditions for admissibility}
 Let $S$ be a ground ring with
parameters $\rho$, $q$, $\delta_j$ ($j \ge 0$)  and $u_1, \dots, u_r$.  
Assume that $(q - q\inv)$ is not a zero--divisor in $S$.
The following conditions are equivalent:
\begin{enumerate}
\item  $\{e_1, y_1 e_1, \dots, y_1^{r-1} e_1\} \subseteq W_2$ is linearly independent over $S$.
\item  The parameters satisfy the following relations:

\begin{equation} \label{equation: yu wilcox admissibility condition 1}
\begin{aligned}
&\rho(a_\ell - a_{r-\ell}/a_0) \ + 
\\& (q-q\inv)\bigg [ \sum_{j = 1}^{r - \ell} a_{j+\ell} \delta_j 
-  \sum_{j = \max(\ell + 1, \lceil r/2 \rceil)}^{\lfloor (\ell + r)/2 \rfloor} a_{2j - \ell}
+   \sum_{j =  \lceil \ell/2 \rceil}^{\min(\ell, \lceil r/2 \rceil -1)} a_{2j - \ell} \bigg ]= 0, \\& \quad \text{for $1 \le \ell \le r-1$},
\end{aligned}
\end{equation}
\begin{equation} \label{equation: yu wilcox admissibility condition 2}
\rho\inv a_0 - \rho a_0\inv = 
\begin{cases}
0 & \text{if  $r$ is odd} \\
(q - q\inv) & \text{if $r$ is even},
\end{cases}
\end{equation}
 and
\begin{equation} \label{equation:  weak admissibility in yu wilcox}
\delta_a = -\sum_{j=0}^{r-1} a_j \delta_{a-r+j} \quad \text{for }  a \ge r.
\end{equation}

\item  $S$ is weakly admissible, and $W_2$ admits a module $M$ with an $S$--basis \break $\{v_0, y_1 v_0, \dots,  y_1^{r-1} v_0\}$  such that  $e_1 (y_1^j v_0) = \delta_j v_0$\  for \ 
$0 \le j \le  r-1$, \  $g_1 v_0 = \rho\inv v_0$, and \ $y_2  y_1^j v_0 = y_1^{j-1} v_0 $.     
\end{enumerate}
\end{theorem}

\begin{definition}[Wilcox and Yu, \cite{Wilcox-Yu}]     Let $S$ be a  ground ring with
parameters $\rho$, $q$, $\delta_j$ ($j \ge 0$)  and $u_1, \dots, u_r$.  Assume that $(q - q\inv)$ is not a zero--divisor in $S$.
One says that $S$ is {\em admissible} (or that the parameters are {\em admissible})  if the equivalent conditions of Theorem \ref{theorem: equivalent conditions for admissibility} hold.
\end{definition}

\begin{remark} \mbox{}
\begin{enumerate}
\item In later work,  Wilcox and Yu considered a  more subtle version of their admissibility condition that is also valid if $q - q\inv$ is a zero--divisor.  
\item  If $R$ is an integral ground ring and Equation (\ref{equation: yu wilcox admissibility condition 2}) holds, then $\rho = \pm a_0$ if $r$ is odd, and $\rho \in \{q\inv a_0,  -q a_0\}$ if $r$ is even.
\end{enumerate}
\end{remark}

Next we discuss the admissiblity condition of Rui and Xu ~\cite{rui-2008},   called   {\em $u$-admissibility}.   In ~\cite{rui-2008},  ground rings are assumed to be integral domains, and it is assumed that $q - q\inv$ is invertible.  Since we do not want to specialize to this situation immediately,  the form in which we describe $u$--admissibility will be a little different from that in ~\cite{rui-2008}. 

The definition of $u$--admissibility is based on a heuristic involving linear independence of
 $\{e_1, y_1 e_1, \dots, y_1^{r-1} e_1\} \subseteq W_2$, under additional assumptions on $u_1, \dots, u_r$.
 Suppose that $F$ is a field and $u_1, \dots, u_r$ are {\em distinct}  invertible elements of $F$ with
 $u_i u_j \ne 1$ for all $i,  j$.   Moreover, suppose $\rho$ and $q$ are non-zero elements of $F$ with
 $q - q\inv \ne 0$.   Define quantities $\gamma_j$  ($1 \le j \le r$)  by
 \begin{equation}\label{equation: formula for gammas00}
\gamma_j =\prod_{\ell \ne j} \frac{(u_\ell u_j -1)}{u_j - u_\ell} \left(  
 \frac{1-u_j^2}{\rho(q\inv -q)}  \prod_{\ell \ne j} u_\ell   + 
\begin{cases}
1 &\text{ if $r$  is odd} \\
-u_j     &\text{if $r$ is even} 
\end{cases}
\right )
\end{equation}
The elements $\gamma_j$  arise as the unique solutions to the system of linear equations:
\begin{equation}\label{equation:  equation for gammas00}
\sum_j  \frac{1}{1-u_i u_j}\,  \gamma_j =    \frac{1}{1-u_i^2}  +  \frac{1}{\rho (q\inv - q)} \quad (1 \le i \le r)
\end{equation}

Then one has the following analogue of the theorem of Wilcox \& Yu cited above:
\begin{theorem}[\cite{GH3}, Theorem 3.10] \label{theorem: equivalent conditions for u--admissibility}
 Let $S$ be an integral ground ring with parameters $\rho$, $q$, $\delta_j$ ($j \ge 0$)  and $u_1, \dots, u_r$.
    Assume that $(q - q\inv) \ne 0$, that   the elements $u_i$ are distinct,  and  that $u_i u_j \ne 1$ for all $i, j$.
 Define $\gamma_j$ in the field of fractions of $S$ by  (\ref{equation: formula for gammas00}), for $1 \le j \le r$.
 Then the following conditions are equivalent:

\begin{enumerate}

\item  $\{e_1, y_1 e_1, \dots, y_1^{r-1} e_1\} \subseteq  \bmw {2,S}$ is linearly independent over $S$.

\item    For all $a \ge 0$, we have $\delta_a = \sum_{j = 1}^r  \gamma_j u_j^a$.

\end{enumerate}
\end{theorem}  

Of course, by Theorem \ref{theorem: equivalent conditions for admissibility}, the conditions are equivalent to the admissibility of $S$ (in the special case considered, namely that the $u_i$ are distinct and 
$u_i u_j \ne 1$ or all $i, j$.)

\medskip 
Although the $\gamma_j$ are rational functions with singularities at $u_i = u_j$,   one can show that the quantities $(q-q\inv)\sum_{j = 1}^r  \gamma_j u_j^a$ are
polynomials   in $u_1, \dots, u_r$, $\rho\powerpm$, and  \break $(q-q\inv)$, as follows:
Let $\ubold_1,  \dots, \ubold_r$,   $\rhobold$,  $\qbold$, and   $t$  be algebraically independent indeterminants over $\Z$.  
Define
\begin{equation}\label{equation: definition of G(t)}
G(t) = G(\ubold_1, \dots, \ubold_r; t) = \prod_{\ell = 1}^r  \frac{t - \ubold_\ell}{t \ubold_\ell -1}.
\end{equation}
Let $\mu_a = \mu_a(\ubold_1, \dots, \ubold_r)$  denote the $a$--th coefficient of the formal power series expansion of $G(t)$.
Notice that each $\mu_a$  is a symmetric polynomial in $\ubold_1,  \dots, \ubold_r$  and that
$G(t\inv)  =  G(t)\inv$.  Define   
\begin{equation} \label{equation: formula for Z of t}
Z(t) = Z(t; \ubold_1, \dots, \ubold_r, \rhobold, \qbold) = -\rhobold\inv   +  (\qbold -\qbold\inv)\frac{t^2}{t^2 -1} + A(t) \  G(t\inv),
\end{equation}
where
$$ 
A(t) =
\begin{cases}
\displaystyle
{- \rho\inv a_0 }\ +\  (\qbold -\qbold\inv){t}/{(t^2-1)}  \quad &\text{if $r$ is  odd, and } \\
\displaystyle
{\rho\inv   a_0} \ -\  (\qbold -\qbold\inv){t^2}/{(t^2-1)} \quad &\text{if $r$ is  even}.
\end{cases}
$$

In the following, we use the notation  $\delta_{(P)} = 1$ if $(P)$ is true and $\delta_{(P)} = 0$ if $(P)$ is false.
 Write $a_j$  for $a_j(\ubold_1, \dots, \ubold_r) = (-1)^{r-j} \eps_j(\ubold_1, \dots, \ubold_r)$.  

\vbox{
\begin{lemma}[\cite{GH3},  Lemma  3.18;  ~\cite{rui-2008}, Lemmas 2.23 and 2.28]   \label{lemma:  calculation of Z1} \mbox{}

 Let  $\ubold_1,  \dots, \ubold_r$,   $\rhobold$,  $\qbold$, and   $t$  be algebraically independent indeterminants over $\Z$.   Define  $\eta_a = \sum_j  \gamma_j \ubold_j^a$  for $a \ge 0$,  where $\gamma_j$ is given by   (~\ref{equation: formula for gammas00}).    Then
\begin{enumerate}
\item
$\displaystyle (\qbold -\qbold\inv) \ \sum_{a \ge 0} \eta_a t^{-a}  =  Z(t;  \ubold_1, \dots, \ubold_r, \rhobold, \qbold)$.
\suspend{enumerate}

 Now let $R$ be any commutative ring with invertible  elements $\rho$, $q$, and $u_1, \dots, u_r$, and additional elements $\eta_a$,  $a \ge 0$.    Let $t$ be an indeterminant over $R$.  Let \  $\mu_a = \mu_a(u_1, \dots, u_r)$  be the coefficients of the  formal power series expansion of 
 $G(u_1, \dots, u_r; t)$. 
Suppose that
$$
(q - q\inv) \sum_{a \ge 0} \eta_a t^{-a}  = Z(t; u_1, \dots, u_r, \rho, q).
$$
 Then
\resume{enumerate}
\item
 If $r$ is odd, then
for $a \ge 0$,  
$$
\begin{aligned}
(q -q\inv)\eta_a = - \delta_{(a = 0)}\ & {\rho\inv  }
+(q -q\inv) \delta_{(\text{$a$ is even})}  \\ & -     \mu_a \,{\rho\inv a_0 }  + 
(q -q\inv)(\mu_{a-1}  + \mu_{a-3}  + \cdots ) .
\end{aligned}
$$
\item
If $r$ is even,  then
for $a \ge 0$,  
$$
\begin{aligned}
(q -q\inv)\eta_a = - \delta_{(a = 0)}\ & {\rho\inv  }
+ (q -q\inv)\delta_{(\text{$a$ is even})}  \\ &+     \mu_a  \,{\rho\inv a_0 } - (q -q\inv)(\mu_a   + \mu_{a-2}  + \mu_{a-4}  + \cdots)  .
\end{aligned}
$$
\item  $\displaystyle (q -q\inv)\, \eta_0 =  ({a_0^2 - 1 })\,\rho\inv
+(q -q\inv)\,( 1  - \delta_{(\text{$r$ is even})}\ a_0)$.
\item  For all $a \ge 0$,   $(q- q\inv) \eta_a$ is an element of the ring
 $\Z[u_1,  \dots, u_r, q - q\inv,  \rho\inv],$
 and is symmetric in $u_1, \dots, u_r$.
\end{enumerate}
\end{lemma}
}

\begin{remark}
In the lemma, it is not assumed that we are working in a ground ring, i.e. that condition
(\ref{equation:  basic relation in ground ring}) holds.  
\end{remark}

Suppose that $S$ is an integral ground ring in which the $u_j$ are distinct,  $u_i u_j \ne 1$ for all $i,  j$, and $q - q\inv \ne 0$. Suppose, moreover, that $S$ is admissible, that is   \break 
$\{e, y_1 e, \dots, y_1^{r-1} e\} \subseteq  \bmw {2,S}$ is linearly independent over $S$.    Then by Theorem \ref{theorem: equivalent conditions for u--admissibility},  we have $\delta_a = \sum_{j = 1}^r \gamma_j u_j^a$  for $a \ge 0$.   It then follows from Lemma \ref{lemma:  calculation of Z1}, part (1), that
\begin{equation} \label{Z is gen fn for delta's}
(q - q\inv) \sum_{a \ge 0} \delta_a t^{-a}  = Z(t; u_1, \dots, u_r, \rho, q).
\end{equation}
 However, Equation (\ref{Z is gen fn for delta's}) makes sense as a condition on ground rings, without any special assumptions on the elements $u_i$;  this motivates the following definition of Rui and Xu:

\begin{definition}[Rui and Xu, \cite{rui-2008}]  \label{definition: RX admissibility}
 Let $S$ be a ground ring with
parameters $\rho$, $q$, $\delta_j$ ($j \ge 0$)  and $u_1, \dots, u_r$.  
Assume that $(q - q\inv)$ is not a zero--divisor in $S$. One says that 
$S$ is {\em $u$--admissible} (or that the parameters are {\em $u$--admissible})   if  
$$(q - q\inv) \sum_{a \ge 0}  \delta_a t^{-a} =    Z(t; u_1, \dots, u_r, \rho, q),$$  
where $Z$ is defined in Equation (\ref{equation: formula for Z of t}).
\end{definition}

\begin{remark} \label{remark: parameter preserving transformations preserve admissibility}
 \mbox{}
 \begin{enumerate}
 \item Suppose that   $S$ is a $u$--admissible ground ring.
 Then conclusions  (2)--(5) of Lemma \ref{lemma:  calculation of Z1} hold, with  $\eta_a$ replaced with
 $\delta_a$.   Moreover, statement (4) of  Lemma  \ref{lemma:  calculation of Z1} together with the ground ring condition (\ref{equation:  basic relation in ground ring}) implies that condition (\ref{equation: yu wilcox admissibility condition 2}) holds.    If, in addition, $R$ is assumed to be integral, then
we have   $\rho = \pm a_0$ if $r$ is odd, and $\rho \in \{q\inv a_0,  -q a_0\}$ if $r$ is even.
  \item  

Let $S$ be a ground ring with admissible (resp. $u$--admissible)  parameters  $\rho$, $q$, $\delta_j$   ($j \ge 0$),  and
$u_1, \dots, u_r$.  Then  
$$
\rho, -q\inv, \delta_j \  (j \ge 0),  \text{ and }  u_1, \dots, u_r
$$
and
$$
-\rho, -q, \delta_j \  (j \ge 0),  \text{ and }  u_1, \dots, u_r
$$
are also sets of admissible (resp. $u$--admissible) parameters.  
 If $S$ is a  ground ring with admissible (resp. $u$--admissible)  parameters and $\varphi : S \rightarrow S'$ is 
a morphism of ground rings in the sense of Definition \ref{definition: parameter preserving}, such that
 $\varphi(q - q\inv)$  is  not a zero--divisor,  then $S'$ is also admissible (resp. $u$--admissible).
 \item  Considering parts (1) and (2) of this remark,  if $S$ is a $u$--admissible integral ground ring, one can 
 assume $\rho = -a_0 = \prod_{j = 1}^r u_j$,   if $r$ is odd, and $\rho = q\inv a_0 = q\inv \prod_{j = 1}^r u_j$,  if $r$ is even. 
\end{enumerate}
\end{remark}

\ignore{
We now consider the case  that ground rings are integral domains.
Equation (\ref{equation: yu wilcox admissibility condition 2}),  which appears in the definitions of admissibility and of $u$--admissiblity,   is equivalent to
\begin{equation}
\begin{cases}
(a_0 - \rho)(a_0 + \rho) = 0 &\text{if $r$ is odd, and}\\
(\rho - q\inv a_0)(\rho + q a_0) = 0 &\text{if $r$ is even.}\\
\end{cases}
\end{equation}
If $S$ is an integral domain, we have  $\rho = \pm a_0$ if $r$ is odd,  and $\rho = -q a_0$ or $\rho = q\inv a_0$ if $r$ is even.  However,  we do not need to entertain both solutions.  Considering Remark \ref{remark: parameter preserving transformations preserve admissibility},  we can assume
\begin{equation} \label{equation: solve for rho 1}
\rho = -a_0 =\prod_{j = 1}^r u_j, \quad \text{if $r$ is odd, and}
\end{equation}
\begin{equation} \label{equation: solve for rho 2}
\rho = q\inv a_0 = q\inv \prod_{j = 1}^r u_j, \quad \text{if $r$ is even}.
\end{equation}
Using this,  we can eliminate $\rho$  from equation (\ref{equation: yu wilcox admissibility condition 2}) in the definition of admissibility and from equation (\ref{equation:  RX admissibility first condition}) in the definition of $u$--admissibility.   
}

\section{Equivalence of admissibility and $u$--admissibility}

  Let $\ubold_1,  \dots, \ubold_r$,   $\rhobold$,  $\qbold$, and   $t$  be algebraically independent indeterminants over $\Z$.    Define $Z(t) \in \Q(\ubold_1,  \dots, \ubold_r, \rhobold, \qbold, t)$  by
  Equation (\ref{equation: formula for Z of t}), 
 and define $\eta_a$  for $a \ge 0$ by 
 $$
 (\qbold - \qbold\inv) \sum_{a \ge 0} \eta_a t^{-a} = Z(t;  \ubold_1, \dots, \ubold_r, \rhobold, \qbold).
 $$   
 Then   statements
 (2)--(5)   of Lemma \ref{lemma:  calculation of Z1} hold;  in particular, by 
 part (5) of Lemma \ref{lemma:  calculation of Z1},  $(\qbold -  \qbold\inv) \eta_a \in
\Z[\ubold_1,  \dots, \ubold_r, \qbold - \qbold\inv,  \rhobold\inv]$.   

\begin{lemma}   The elements $\eta_j $ satisfy
\begin{equation} \label{equation:  admissibility  of  eta's}
\begin{aligned}
 \big [\rhobold\inv &a_0 - \delta_{(\text{$r$ is even})}(\qbold - \qbold\inv) \big ](a_0 a_\ell - a_{r-\ell}  )  \\& 
 +  (\qbold\   -\qbold\inv)    \left (  \sum_{j = 1}^{r - \ell}  \eta_j a_{j+\ell}  
-  \sum_{j = \max(\ell + 1, \lceil r/2 \rceil)}^{\lfloor (\ell + r)/2 \rfloor} a_{2j - \ell}
+   \sum_{j =  \lceil \ell/2 \rceil}^{\min(\ell, \lceil r/2 \rceil -1)} a_{2j - \ell} \right ) = 0,  
\end{aligned}
\end{equation}
for $1\le  \ell \le r-1$.
\end{lemma}

\begin{proof} 
We have
\begin{equation} \label{equation: calculation of u admissibility 0}
\begin{aligned}
(\qbold\  -\qbold\inv) & \sum_{a \ge 0} \eta_a t^a = Z(t\inv) 
\\&= - \rhobold\inv + (\qbold\  -\qbold\inv)/(1-t^2)   + A(t\inv)  \prod_{\ell = 1}^r  \frac{t - \ubold_\ell}{t \ubold_\ell -1}.
\end{aligned}
\end{equation}
Multiplying both sides of  (\ref{equation: calculation of u admissibility 0}) by $\prod_{\ell=1}^r (t\ubold_\ell - 1)$ gives
\begin{equation} \label{equation: calculation of u admissibility 1}
\begin{aligned}
(\qbold\  &-\qbold\inv) \ \prod_{\ell = 1} ^r (t\ubold_\ell - 1) \ \sum_{a \ge 0}  \eta_a t^{a}  \\&= - \rhobold\inv  \prod_{\ell = 1} ^r (t\ubold_\ell - 1) \ +  \ \frac{ \qbold -\qbold\inv}{1-t^2}  \prod_{\ell = 1} ^r (t\ubold_\ell - 1)\ + \ A(t\inv) \  \prod_{\ell = 1} ^r (t - \ubold_\ell) .
 \end{aligned}
\end{equation}
For $1 \le \ell \le r-1$, the coefficient of $t^{r - \ell}$ on the left side of (\ref{equation: calculation of u admissibility 1}) is
\begin{equation} \label{equation: calculation of u admissibility 2}
(-1)^r \ (\qbold  -\qbold\inv) \left (\eta_0 a_\ell +       \sum_{j = 1}^{r-\ell}  \eta_j a_{j+\ell}\right ).
\end{equation}
Taking into account the formula for  $\eta_0$ in part (4) of Lemma \ref{lemma:  calculation of Z1}, (\ref{equation: calculation of u admissibility 2}) becomes 
\begin{equation} \label{equation: calculation of u admissibility 3}
\begin{aligned}
(-1)^r \  &\bigg ( ({a_0^2 - 1 })\rhobold\inv a_\ell
+(\qbold -\qbold\inv) a_\ell  
  \\ & - \delta_{(\text{$r$ is even})} (\qbold -\qbold\inv) a_0  a_\ell 
+     (\qbold  -\qbold\inv)  \sum_{j = 1}^{r-\ell}  \eta_j a_{j+\ell}\bigg ).
\end{aligned}
\end{equation}
Now suppose that $r$ is odd.  Then for $1 \le \ell \le r-1$,  the coefficient of $t^{r-\ell}$ on the right side of (\ref{equation: calculation of u admissibility 1}) is
\begin{equation} \label{equation: calculation of u admissibility 4}
\begin{aligned}
\rhobold\inv a_\ell    & - \rhobold\inv a_0 a_{r-\ell}  \\& 
 -   (\qbold-\qbold\inv) \sum_{i = 0}^{\lfloor(r-\ell)/{2}\rfloor} a_{\ell + 2 i}  
   \ +  \  (\qbold-\qbold\inv) \sum_{i = 0}^{\lfloor(r-1-\ell)/{2}\rfloor} a_{r-1-\ell - 2 i} .
\end{aligned}
\end{equation}
Continuing with the case that $r$ is odd, and setting (\ref{equation: calculation of u admissibility 3})  equal to (\ref{equation: calculation of u admissibility 4}), we get
\begin{equation} \label{equation: calculation of u admissibility 5}
\begin{aligned}
0 = 
&  \rhobold\inv a_0(a_0 a_\ell - a_{r-\ell}  ) \\& 
 +  (\qbold-\qbold\inv) \left ( \sum_{j = 1}^{r-\ell}  \eta_j a_{j+\ell}  \  + \ a_\ell -  \sum_{i = 0}^{\lfloor(r-\ell)/{2}\rfloor} a_{\ell + 2 i}  
   \  +  \  \sum_{i = 0}^{\lfloor(r-1-\ell)/{2}\rfloor} a_{r-i-\ell - 2 i}  \right ).
\end{aligned}
\end{equation}
By examining cases, according to the parity of $\ell$ and the sign of   $\ell + 1 - \lceil r/2 \rceil$,  one can check that  the expression on the second line of (\ref{equation: calculation of u admissibility 5})  is equal to 
\begin{equation} \label{equation: calculation of u admissibility 6}
\begin{aligned}
 (\qbold\   -\qbold\inv)    \left (  \sum_{j = 1}^{r - \ell}  \eta_j a_{j+\ell}  
-  \sum_{j = \max(\ell + 1, \lceil r/2 \rceil)}^{\lfloor (\ell + r)/2 \rfloor} a_{2j - \ell}
+   \sum_{j =  \lceil \ell/2 \rceil}^{\min(\ell, \lceil r/2 \rceil -1)} a_{2j - \ell} \right ).
   \end{aligned}
\end{equation}
For example,  if $\ell$ is odd and $\ell + 1 \le   \lceil r/2 \rceil = (r+1)/2$,  then

\begin{equation} \label{equation: calculation of u admissibility 7}
\begin{aligned}
& -  \sum_{j = \max(\ell + 1, \lceil r/2 \rceil)}^{\lfloor (\ell + r)/2 \rfloor} a_{2j - \ell}
+   \sum_{j =  \lceil \ell/2 \rceil}^{\min(\ell, \lceil r/2 \rceil -1)} a_{2j - \ell} 
\\& = 
   \sum \{ a_k | k \text{ odd and }  1 \le k \le \ell \} -
      \sum \{ a_k | k \text{ odd and }  r+ 1 - \ell  \le k \le r \}, 
  \end{aligned}
  \end{equation}
  while 
  \begin{equation}   \label{equation: calculation of u admissibility 8}
\begin{aligned}
& a_\ell - \sum_{i = 0}^{\lfloor(r-\ell)/{2}\rfloor} a_{\ell + 2 i}  
   \  +  \  \sum_{i = 0}^{\lfloor(r-1-\ell)/{2}\rfloor} a_{r-i-\ell - 2 i} 
   \\& = 
   \sum \{ a_k | k \text{ odd and }  1 \le k \le r - 1 - \ell \} -
      \sum \{ a_k | k \text{ odd and }  \ell + 2  \le k \le r \}. 
      \end{aligned}
\end{equation}
The summands  $\{a_k | k \text{ odd  and } \ell + 2  \le k \le r - 1 - \ell \}$ appear in both of the sums on the last line,  so they cancel to give
 \begin{equation}   \label{equation: calculation of u admissibility 9}
\begin{aligned}
& a_\ell - \sum_{i = 0}^{\lfloor(r-\ell)/{2}\rfloor} a_{\ell + 2 i}  
   \  +  \  \sum_{i = 0}^{\lfloor(r-1-\ell)/{2}\rfloor} a_{r-i-\ell - 2 i} 
   \\& = 
   \sum \{ a_k | k \text{ odd and }  1 \le k \le  \ell \} -
      \sum \{ a_k | k \text{ odd and }  r + 1 - \ell  \le k \le r \}. 
      \end{aligned}
\end{equation}
Comparing (\ref{equation: calculation of u admissibility 7})  and (\ref{equation: calculation of u admissibility 9})
gives 
 \begin{equation}   \label{equation: calculation of u admissibility 10}
\begin{aligned}
 a_\ell - \sum_{i = 0}^{\lfloor(r-\ell)/{2}\rfloor} a_{\ell + 2 i}  
  & \  +  \  \sum_{i = 0}^{\lfloor(r-1-\ell)/{2}\rfloor} a_{r-i-\ell - 2 i} 
 \\&   = 
   -  \sum_{j = \max(\ell + 1, \lceil r/2 \rceil)}^{\lfloor (\ell + r)/2 \rfloor} a_{2j - \ell}
+   \sum_{j =  \lceil \ell/2 \rceil}^{\min(\ell, \lceil r/2 \rceil -1)} a_{2j - \ell}, 
         \end{aligned}
\end{equation}
and therefore  the second line of (\ref{equation: calculation of u admissibility 5})  is equal to(\ref{equation: calculation of u admissibility 6}).   The other cases are handled similarly. 
This completes the proof of the lemma when $r$ is odd.

Now consider the case that $r$ is even.   Then  for $1 \le \ell \le r-1$,  the coefficient of $t^{r-\ell}$ on the right side of (\ref{equation: calculation of u admissibility 1}) is
\begin{equation} \label{equation: calculation of u admissibility 11}
\begin{aligned}
- \rhobold\inv a_\ell    & + \rhobold\inv a_0 a_{r-\ell}  \\& 
 +   (\qbold-\qbold\inv) \sum_{i = 0}^{\lfloor(r-\ell)/{2}\rfloor} a_{\ell + 2 i}  
   \ -  \  (\qbold-\qbold\inv) \sum_{i = 0}^{\lfloor(r-\ell)/{2}\rfloor} a_{r-\ell - 2 i} .
\end{aligned}
\end{equation}
Setting (\ref{equation: calculation of u admissibility 3})  equal to (\ref{equation: calculation of u admissibility 11}), we get
\begin{equation} \label{equation: calculation of u admissibility 12}
\begin{aligned}
0 = 
&  \rhobold\inv a_0^2 a_\ell -  \rhobold\inv a_0 a_{r-\ell}   - (\qbold - \qbold\inv) a_0 a_\ell \\& 
 +  (\qbold-\qbold\inv) \left ( \sum_{j = 1}^{r-\ell}  \eta_j a_{j+\ell}  \  + \ a_\ell -  \sum_{i = 0}^{\lfloor(r-\ell)/{2}\rfloor} a_{\ell + 2 i}  
   \  +  \  \sum_{i = 0}^{\lfloor(r-\ell)/{2}\rfloor} a_{r-\ell - 2 i}  \right ) 
   \\ &  = \big [\rhobold\inv a_0 - (\qbold - \qbold\inv) \big ](a_0 a_\ell - a_{r-\ell})  
   \\& 
 +  (\qbold-\qbold\inv) \left ( \sum_{j = 1}^{r-\ell}  \eta_j a_{j+\ell}  \  + \ a_\ell  - a_{r-\ell} -  \sum_{i = 0}^{\lfloor(r-\ell)/{2}\rfloor} a_{\ell + 2 i}  
   \  +  \  \sum_{i = 0}^{\lfloor(r-\ell)/{2}\rfloor} a_{r-\ell - 2 i}  \right ). 
\end{aligned}
\end{equation}
As in the case that $r$ is odd,  one can show that the expression in the last line of (\ref{equation: calculation of u admissibility 12})  is equal to (\ref{equation: calculation of u admissibility 6}).   This completes the proof in case
$r$ is even.  
\end{proof}  

\begin{corollary} \label{corollary:  eta's satisfy Wilcox Yu condition}

 Let  $$\Lambda =\Z[\ubold_1\powerpm,  \dots, \ubold_r\powerpm, \rhobold\powerpm, \qbold\powerpm, (\qbold - \qbold\inv)\inv ]/ I$$
where
$I$ is the ideal generated by $  \rhobold\inv a_0-  \rhobold a_0\inv  - \delta_{(\text{$r$ is even})} (\qbold - \qbold\inv) $.    The image of the elements $\eta_j$  in $\Lambda$ satisfy
\begin{equation} \label{equation:  admissibility  of  eta's 2}
\begin{aligned}
\rhobold( a_\ell - & a_{r-\ell}/a_0  )  \\& 
 +  (\qbold\   -\qbold\inv)    \left (  \sum_{j = 1}^{r - \ell}  \eta_j a_{j+\ell}  
-  \sum_{j = \max(\ell + 1, \lceil r/2 \rceil)}^{\lfloor (\ell + r)/2 \rfloor} a_{2j - \ell}
+   \sum_{j =  \lceil \ell/2 \rceil}^{\min(\ell, \lceil r/2 \rceil -1)} a_{2j - \ell} \right ) = 0,  
\end{aligned}
\end{equation}
for $1\le  \ell \le r-1$.
\end{corollary}

\begin{lemma} \label{lemma: weak admissibility condition for eta's}
 For $m \ge r$,  one has $\sum_{j= 0}^r  a_j \eta_{j+m-r} = 0$.  
\end{lemma}

\begin{proof}   For $m \ge r$, the coefficient of $t^m$ on the left side of (\ref{equation: calculation of u admissibility 1})  is 
$$(-1)^r (\qbold - \qbold\inv) \sum_{j = 0}^r a_j \eta_{j + m-r}.$$
Thus, we have to show that the coefficient of $t^m$ on the right side of    
(\ref{equation: calculation of u admissibility 1})  is zero. 

If $r$ is odd, then the right side of (\ref{equation: calculation of u admissibility 1}) is
\begin{equation} \label{equation:  calculation of weak admissibility for eta's 1}
\begin{aligned}
-\rhobold\inv \prod_{\ell = 1}^r (t u_\ell -1)  &- \rhobold\inv a_0 \prod_{\ell = 1}^r (t - u_\ell)  \\&
+ (\qbold - \qbold\inv) \left ( \frac{\prod_{\ell = 1}^r (t u_\ell -1)}{1-t^2}     +        
\frac{t \prod_{\ell = 1}^r (t  -u_\ell)}{1-t^2}             \right )
\end{aligned}
\end{equation}
For $m > 0$, the coefficient of $t^m$ in the first line of (\ref{equation:  calculation of weak admissibility for eta's 1}) is zero.  Moreover, the coefficient of $t^r$ is 
$-\rhobold\inv ( \prod_{\ell = 1}^r \ubold_\ell  - a_0) = 0$.

Write $a_k = 0$  if $k < 0$  or $k > r$.  Then the second line of  (\ref{equation:  calculation of weak admissibility for eta's 1}) expands to 
\begin{equation} \label{equation:  calculation of weak admissibility for eta's 2}
\begin{aligned}
&(\qbold - \qbold\inv)\left ( (-1)^r \bigg [\sum_{j = 0}^r t^j a_{r-j}\bigg ] 
\bigg [\sum_{\ell \ge 0} t^{2 \ell}  \bigg ]   +   
\bigg [\sum_{j = 0}^r t^j a_{j}\bigg ] 
  \bigg [\sum_{\ell \ge 0} t^{2 \ell + 1}  \bigg ] \right )
  \\& = (\qbold - \qbold\inv)\left ( (-1)^r \sum_{m \ge 0}   \bigg [\sum_{\ell \ge 0} a_{r-m+ 2\ell} \bigg ] t^m
  + \sum_{m \ge 0}   \bigg [ \sum_{\ell \ge 0} a_{m-1 - 2 \ell} \bigg ] t^m
   \right )
\end{aligned}
\end{equation}
For $m \ge r$,  the coefficient of $t^m$  in (\ref{equation:  calculation of weak admissibility for eta's 2}) is zero.    Thus,  for $m \ge r$,  the coefficient of $t^m$ in (\ref{equation:  calculation of weak admissibility for eta's 1})  is zero.  

The proof when $r$ is even is similar.
\end{proof}

\begin{theorem} \label{theorem: equivalence}
Let $S$ be a ground ring with parameters $\rho$, $q$, $\delta_j$ and $u_1, \dots, u_r$, with $q - q\inv$ not a zero--divisor.  Then $S$ is admissible if and only if $S$ is $u$--admissible.
\end{theorem}

 \begin{proof}  Let $\eta_a$ ($a \ge 0$)  be determined by $$(q - q\inv ) \sum_{a \ge 0} \eta_a t^{-a} =
 Z(t; u_1, \dots, u_r, \rho, q),$$

Suppose that the parameters are $u$--admissible.   Then
 $\delta_a = \eta_a$ for $a \ge 0$ by definition of $u$--admissibility, and the assumption on $q - q\inv$.   Condition (\ref{equation: yu wilcox admissibility condition 2}) holds by Remark \ref{remark: parameter preserving transformations preserve admissibility}, part (1), and because of this,   
  it follows from Corollary \ref{corollary:  eta's satisfy Wilcox Yu condition}  that  the parameters  satisfy condition (\ref{equation: yu wilcox admissibility condition 1}).   Moreover, the parameters satisfy
 condition (\ref{equation:  weak admissibility in yu wilcox})  according to Lemma \ref{lemma: weak admissibility condition for eta's}.   Thus the parameters are admissible.

 Conversely,  suppose that the  parameters are admissible.   The admissibility conditions (\ref{equation: yu wilcox admissibility condition 1}) and (\ref{equation:  weak admissibility in yu wilcox}) and the ground ring condition (\ref{equation:  basic relation in ground ring}) 
 uniquely determine the quantities  $(q-q\inv) \delta_a$ for  ($a \ge 0$)  as Laurent polynomials
 in $\rho$ and $u_1, \dots, u_r$.   Indeed,   note that 
 (\ref{equation: yu wilcox admissibility condition 1}) is a  system of linear equations in the variables $(q - q\inv) \delta_j$ ($1 \le j \le r-1$)   with unitriangular matrix of coefficients.    (Compare  ~\cite{GH3},  Remark 3.7.)     For $a \ge r$, the weak admissibility condition
 (\ref{equation:  weak admissibility in yu wilcox}), determines $\delta_a$ as a polynomial
 in $u_1, \dots, u_r$  and $\{ \delta_j : j < a\}$.   Finally (\ref{equation:  basic relation in ground ring})  determines  $(q - q\inv) \delta_0$.    
 
 Consider the new set of parameters $P' = (\rho, q, \eta_a, u_1, \dots, u_r)$  with the $\delta_a$'s replaced by the $\eta_a$'s and the other parameters unchanged.    We claim that $P'$ is also a set
 of admissible parameters (satisfying the ground ring condition). In fact,  condition  (\ref{equation: yu wilcox admissibility condition 2}) holds for $P'$, because it involves only $\rho$, $q$ and $u_1, \dots, u_r$.    The ground ring condition (\ref{equation:  basic relation in ground ring}) for $P'$ follows from condition  (\ref{equation: yu wilcox admissibility condition 2}) and  Lemma
 \ref{lemma:  calculation of Z1} part (4).  $P'$ satisfies
 conditions (\ref{equation: yu wilcox admissibility condition 1})  and 
 (\ref{equation:  weak admissibility in yu wilcox})  by Corollary \ref{corollary:  eta's satisfy Wilcox Yu condition} and Lemma \ref{lemma: weak admissibility condition for eta's}.   This finishes the verification of the claim.
 
 Since, $P'$ is a set of admissible parameters, 
 the  quantities $(q-q\inv) \eta_a$ for  ($a \ge 0$) are given by the same Laurent polynomials in the remaining parameters as are the quantities $(q-q\inv) \delta_a$ for  ($a \ge 0$).   Since $q - q\inv$ is not a zero divisor, we 
 have $\delta_a = \eta_a$  for all $a \ge 0$, and hence
 the original parameters are $u$--admissible.   
 \end{proof}
\bibliographystyle{amsplain}
\bibliography{admissibility} 

\end{document}